\documentclass[final,leqno]{siamltex}

\usepackage{amsmath,amssymb}

\title{On the representability of the bi-uniform 
matroid\thanks{Research was (partially)
completed while the first author was 
visiting the Institute for Mathematical Sciences,
National University of Singapore in 2011.}}

\author{Simeon Ball\footnotemark[2]
\and Carles Padr\'o\footnotemark[3]
\and Zsuzsa Weiner\footnotemark[4]
\and Chaoping Xing\footnotemark[3]}

\begin{document}

\maketitle

\renewcommand{\thefootnote}{\fnsymbol{footnote}}

\footnotetext[2]{Universitat 
Polit\`ecnica de Catalunya, Barcelona, Spain.
The author acknowledges
the support of the project MTM2008-06620-C03-01
of the Spanish Ministry of Science and Education 
and the project 2009-SGR-01387
of the Catalan Research Council.}
\footnotetext[3]{Nanyang Technological 
University, Singapore.
The work of these authors
was supported by the Singapore National
Research Foundation under 
Research Grant NRF-CRP2-2007-03.}
\footnotetext[4]{E\"otv\"os Lor\'and University,
Budapest, Hungary.
The author was
supported by K 81310 grant and by the 
ERC grant No. 227701 DISCRETECONT.}

\renewcommand{\thefootnote}{\arabic{footnote}}

\begin{abstract}
Every bi-uniform matroid is
representable over all sufficiently large fields.
But it is not known exactly over which finite fields they
are representable, and the existence of
efficient methods to find a representation
for every given bi-uniform matroid 
has not been proved.
The interest of these problems is due to their
implications to secret sharing.
The existence of 
efficient methods to find representations for all
bi-uniform matroids is proved here for the first time.
The previously known efficient constructions
apply only to a particular class of bi-uniform matroids,
while the known general constructions were not proved to be efficient.
In addition, our constructions provide 
in many cases representations
over smaller finite fields.
\end{abstract}



\pagestyle{myheadings}
\thispagestyle{plain}
\markboth{S. BALL, C. PADR\'O, Z. WEINER, AND 
C. XING}{ON THE 
REPRESENTABILITY OF THE BI-UNIFORM MATROID}

\section{Introduction}

Given a class of representable matroids, 
the following are two basic 
questions about the class. 
Over which fields are the 
members of the class representable? 
Are there efficient algorithms to 
construct representations 
for every member of the class?
Here an algorithm is efficient
if its running time is
polynomial in the size of the ground set.
For instance, every transversal matroid is
representable over all sufficiently large
fields~\cite[Corollary~12.2.17]{Oxl92}, 
but it is not known exactly over which 
fields they are representable, and the
existence of efficient algorithms to
construct representations is an
open problem too.

The interest for these problems has been mainly motivated by
their connections to coding theory and cryptology, 
mainly to secret sharing.
Determining over which fields 
the uniform matroids are representable
is equivalent to solving the
Main Conjecture for Maximum Distance Separable Codes.
For more details, and a proof of this conjecture in the prime case,
see~\cite{Ball2011}, and for further information
on when the conjecture is known to hold, see~\cite[Section~3]{HS2001}.
As a consequence of the results by
Brickell~\cite{Bri89}, every representation
of a matroid $M$ over a finite field
provides ideal linear secret sharing schemes for
the access structures that are ports of the matroid $M$.
Because of that, the representability of certain classes
of matroids is closely connected to the search for efficient
constructions of secret sharing schemes
for certain classes of access structures.
The reader is referred to~\cite{MaPa10}
for more information about
secret sharing and its connections to matroid theory.

Several constructions of ideal linear 
secret sharing schemes
for families of relatively simple
access structures with interesting 
properties for the applications have been
proposed~\cite{BTW08,Bri89,FaPa10,FPXY11,HeSa06,Ng06,PaSa00,Sim88,Tas07,TaDy06}.
They are basic and natural generalizations 
of Shamir's~\cite{Sha79}
threshold secret sharing scheme.
A unified approach to all those proposals 
was presented in~\cite{FMP07}.
As a consequence, the open questions about the existence
of such secret sharing schemes for some sizes of the secret value
and the possibility of constructing them efficiently
are equivalent to determining the representability
of some classes of multi-uniform matroids.
See~\cite{FaPa11,FPXY11} 
for more information on this line of work.

In this paper, we analyze the representability
of the bi-uniform matroids.
They were introduced by Ng and Walker~\cite{NgWa01},
but ideal secret sharing schemes for 
the access structures that are determined 
by them were previously presented in~\cite{PaSa00}.
Bi-uniform matroids are defined in terms of their symmetry properties,
specifically the number of clonal classes,
a concept introduced in~\cite{GOVW98}.
Two elements in the ground set of a matroid
are said to be \emph{clones\/} if the map that interchanges them
and fixes all other elements is an automorphism of the matroid.
Being clones is clearly an equivalence relation,
and its equivalence classes are called the \emph{clonal classes\/} of the matroid.
Uniform matroids are precisely those having only one clonal class.
A matroid is said to be \emph{bi-uniform\/} if
it has at most two clonal classes.
Of course, this definition can be generalized to
\emph{$m$-uniform matroids\/} for every positive integer $m$.
A bi-uniform matroid is determined by
its rank, the number of elements in each clonal class,
and the ranks of the two clonal classes,
which are called the \emph{sub-ranks\/} of the bi-uniform matroid.

It is not difficult to check
that every bi-uniform matroid is a transversal matroid, 
and hence it is representable over all
sufficiently large fields.
Moreover, as a consequence of the results in~\cite{FMP07},
every bi-uniform matroid is representable
over all fields with at least ${N \choose k}$ elements,
where $N$ is the size of the ground set
and $k$ is the rank.
The same result applies to tri-uniform 
matroids~\cite{FMP07}, 
but it does not apply to $4$-uniform matroids
because the Vamos matroid is not
representable~\cite[Proposition 6.1.10]{Oxl92}.

Even though the proof in~\cite{FMP07} is constructive,
no efficient method to find representations
for the bi-uniform matroids can be derived from it.
A method to construct a representation
for every bi-uniform matroid was presented 
by Ng~\cite{Ng03}, but it was not proved to be efficient.
Efficient methods to find representations
for the bi-uniform matroids in which
one of the sub-ranks is equal to the rank
can be derived from the
constructions of ideal hierarchical 
secret sharing schemes 
by Brickell~\cite{Bri89} and by Tassa~\cite{Tas07}.
These constructions are analyzed
in Section~\ref{pt:related}.

In this work, we prove for the first time
that there exist efficient algorithms
to find representations for \emph{all\/}
bi-uniform matroids.
In addition, our constructions provide representations
over finite fields that are in many cases
smaller than the ones used in~\cite{Bri89,Ng03,Tas07}.
A detailed comparison is given in Section~\ref{pt:related}.

More specifically, we present three different representations of bi-uniform matroids.
All of them can be obtained in time polynomial in the size of the ground set.
An important parameter in our discussions is  
$d = m + \ell -k$,
where $k$ is the rank of the matroid while
$m$ and $\ell$ are its sub-ranks.
The cases $d = 0$ and $d = 1$ are reduced to the
representability of the uniform matroid.
Our first construction (Theorem~\ref{st:rep4d2}) corresponds to the case $d = 2$,
and we prove that every such bi-uniform matroid
is representable over $\mathbb{F}_q$ if $q$ is odd
and every clonal class has at most $(q-1)/2$ elements.
The other two constructions apply to the general case,
and they are both based on a family of linear evaluation codes.
Our second construction (Theorem~\ref{st:repextfield})
provides a representation of the bi-uniform
matroid over  ${\mathbb F}_{q_0^s}$,
where $s > d(d-1)/2$ and $q_0$ is a prime power
larger than the size of each clonal class.
Finally, we present a third construction in Theorem~\ref{st:repprimefield}.
In this case, if $m \ge \ell$, 
a representation of the bi-uniform matroid
is obtained over every prime field $\mathbb{F}_p$ 
with $p > K^h$, where
$K$ is larger than half the number of elements in
each clonal class and
$h =  md(1 + d(d-1)/2)$.

\section{Related work}
\label{pt:related}

The existence of ideal secret sharing schemes
for the so-called bipartite and tripartite access structures 
was proved in~\cite{PaSa00} and in~\cite{FMP07}, respectively.
These proofs are constructive and, in particular, they
provide a method to find representations 
for all bi-uniform matroids.
Such a representation can be found
over every field with at least
${N \choose k}$ elements,
where $N$ is the size of the ground set
and $k$ is the rank.
This method is not efficient
because exponentially many
determinants have to be computed 
to find a valid representation.

This problem is avoided in the
method proposed by Ng~\cite{Ng03}, 
which provides a representation
for every given bi-uniform matroid.
Specifically, Ng gives a representation for 
the bi-uniform matroid 
with rank $k$ and sub-ranks $m,\ell$
over every finite field of the form
$\mathbb{F}_{q_0^s}$, where $q_0 > 14$,
each clonal class has at most $q_0$ elements,
and $s$ is at least $k$
and co-prime with $d = m + \ell - k$.
This method may be efficient, 
but this fact is not proved in~\cite{Ng03}.
In addition, the degree $s$ of the extension field
depends on the rank $k$, while in our
efficient construction
in Theorem~\ref{st:repextfield}, this degree depends 
only on $d$.
Therefore, if $d$ is small compared to $k$, 
our construction works over smaller fields.
 
Efficient methods to construct 
ideal hierarchical secret sharing schemes were 
given by Brickell~\cite{Bri89} and by Tassa~\cite{Tas07}.
When applied to some particular cases, these
methods provide representations for bi-uniform
matroids in which one of the sub-ranks is equal to the rank.

Brickell's construction provides a representation
for every such bi-uniform matroid 
over fields of the form $\mathbb{F}_{q_0^s}$,
where $q_0$ is a prime power larger 
than the size of each clonal class
and $s$ is at least the square of the rank of the matroid.
An irreducible polynomial 
of degree $s$ over $\mathbb{F}_{q_0}$
has to be found,
but this can be done in time polynomial in $q_0$ and $s$
by using the algorithm given by Shoup~\cite{Sho90}.
Therefore, a representation can be found in
time polynomial in the size of the ground set.
Clearly, the size of the field is much smaller in the
representations that are obtained by the method described
in Theorem~\ref{st:repextfield}.

Representations for those bi-uniform matroids
are efficiently obtained from Tassa's construction over 
prime fields $\mathbb{F}_p$
with $p$ larger than 
$N^{(k-1)(k-2)/2}$, where $N$ 
is the number of elements in the ground set.
If $d$ is small compared to $n$, the size of the field 
in our construction (Theorem~\ref{st:repprimefield})
is smaller.

Representations for bi-uniform matroids in which
one of the sub-ranks is equal to the rank of
the matroid and the other one is equal to $2$
are obtained from the constructions of
ideal hierarchical secret sharing schemes in~\cite{BeWe93}.
These are representations over $\mathbb{F}_q$,
where the size of the ground set
is at most $q+1$ and the size of each clonal class is around $q/2$.
These parameters are similar
to the ones in Theorem~\ref{st:rep4d2},
but our construction is more general.

\section{The bi-uniform matroid}
\label{pt:defs}

A {\em matroid} $M=(E,F)$ is a pair in which $E$
is a finite set, called the \emph{ground set\/}, and
$F$ is a nonempty set of subsets of $E$,
called {\em independent sets}, such that
\begin{enumerate}
\item
every subset of an independent set is an independent subset, and
\item
for all $A \subseteq E$, all maximal independent
subsets of $A$ have the same cardinality,
called the {\em rank of $A$} and denoted $r(A)$.
\end{enumerate}
A {\em basis} $B$ of $M$ is a maximal independent set.
Obviously all bases have the same cardinality, which is called
the \emph{rank of $M$}.
If $E$ can be mapped to a subset of vectors of a vector space over a field ${\mathbb K}$
so that $I \subseteq E$ is an independent set if and only if the vectors assigned
to the elements in $I$ are linearly independent, then the matroid is said to be
{\em representable over ${\mathbb K}$}.

The independent sets of the {\em uniform matroid} of rank $k$
are all the subsets $B$ of the set $E$ with the property that $|B| \leq k$.
If the uniform matroid is representable over a field ${\mathbb K}$ then there is a map
$$
f \ : \ E \rightarrow \ {\mathbb K}^k
$$
such that $f(E)$ is a set of vectors with the property that every subset
of $f(E)$ of size $k$ is a basis of ${\mathbb K}^k$.

For positive integers $k,m,\ell$ with
$1 \le m, \ell \le k$ and $m + \ell \geq k$,
and a partition $E = E_1 \cup E_2$ of the ground set with
$|E_1| \geq m$ and $|E_2| \geq \ell$,
the independent sets of the {\em bi-uniform matroid of rank $k$ and sub-ranks $m$, $\ell$}
are all the subsets $B$ of the ground set with the property that
$|B | \leq k$, $|B \cap E_1| \leq m$ and $|B \cap E_2| \leq \ell$.
Since the maximal independent subsets of $E_1$
have $m$ elements, $r(E_1) = m$.
Similarly, $r(E_2) = \ell$.

If the bi-uniform matroid is representable over a field ${\mathbb K}$ then there is a map
$$
f \ : \ E \rightarrow \ {\mathbb K}^k
$$
such that $f(E)$ is a set of vectors with the property that every subset
$D$ of $f(E)$ of size $k$ with
$|D \cap f(E_1)| \leq m$ and $|D \cap f(E_2)| \leq \ell$
is a basis of ${\mathbb K}^k$.
The dimensions of
$\langle f(E_1) \rangle$
and $\langle f(E_2) \rangle$ are
$m = r(E_1)$ and $\ell = r(E_2)$, respectively.
Thus, if the bi-uniform matroid is representable over ${\mathbb K}$ then we can construct a set
$S \cup T$ of vectors of ${\mathbb K}^k$ such that
$\dim(\langle S \rangle )=m$ and $\dim(\langle T \rangle )=\ell$,
with the property that every subset $B$ of $S \cup T$ of size $k$
with $|B \cap S| \leq m$ and $|B \cap T| \leq \ell$ is a basis.

\section{Necessary conditions}
\label{pt:neccond}

We present here some necessary conditions for a
bi-uniform matroid to be representable over
a finite field ${\mathbb F}_q$.

The following lemma implies that restricting a representation of the bi-uniform matroid
on $E = E_1 \cup E_2$, one gets a representation of the
uniform matroid on $E_1$ of rank $m$ and the uniform matroid on $E_2$ of rank $\ell$.
Therefore, the known necessary conditions for the
representability of the uniform matroid over ${\mathbb F}_q$
can be applied to the bi-uniform matroid.

\begin{lemma}
If $f$ is a map from $E$ to ${\mathbb K}^k$ which gives a representation of the bi-uniform matroid
of rank $k$ and sub-ranks $m$ and $\ell$ then $f(E_1)$ has the property that every
subset of $f(E_1)$ of size $m$ is a basis of $\langle f(E_1) \rangle$.
Similarly, $f(E_2)$ has the property that every subset of $f(E_2)$ of size $\ell$
is a basis of $\langle f(E_2) \rangle$.
\end{lemma}

\begin{proof}
If $L'$ is a set of $m$ vectors of $f(E_1)$ which are linearly dependent then
$L' \cup L$, where $L$ is a set of $k-m$ vectors of $f(E_2)$,
is a set of $k$ vectors of $f(E)$ which do not form a basis of ${\mathbb K}^k$.
\end{proof}

The \emph{dual\/} of a matroid $M$ is the matroid $M^*$
on the same ground set such that its bases
are the complements of the bases of $M$.
Given a representation of $M$ over $\mathbb{K}$,
simple linear algebra operations provide
a representation of $M^*$ over the same
field~\cite[Section~2.2]{Oxl92}.
In particular, if  $\mathbb{K}$ is finite,
a representation of $M^*$ can be
efficiently obtained from a representation of $M$.
By the following proposition,
the dual of a bi-uniform matroid is a bi-uniform matroid
with the same partition of the ground set.

\begin{proposition}
The dual of the bi-uniform matroid of rank $k$ and sub-ranks $m$ and $\ell$
on the ground set $E = E_1 \cup E_2$ is the
bi-uniform matroid of rank $k^* = |E_1|+|E_2|-k$
and sub-ranks $m^* = |E_1| + \ell - k$ and $\ell^* = |E_2|+ m -k$.
\end{proposition}

\begin{proof}
Clearly, a matroid and its dual have the same automorphism group.
This implies that the dual of a bi-uniform matroid is bi-uniform
for the same partition of the ground set.
The values for the rank and the sub-ranks of $M^*$
are derived from the formula that relates the
rank function $r$ of matroid $M$
to the rank function $r^*$ of its dual $M^*$.
Namely,
\(
r^*(A) = |A| - r(E) + r(E \setminus A)
\)
for every $A \subseteq E$~\cite[Proposition~2.1.9]{Oxl92}.
\end{proof}

Clearly, $k = m = \ell$ if and only if
$m^* = |E_1|$ and $\ell^* = |E_2|$,
and in this case both $M$ and $M^*$
are uniform matroids.
We assume from now on that
$m < k$ or $\ell < k$
and that
$m < |E_1|$ or $\ell < |E_2|$,

The results in this paper indicate that
the value $d = m + \ell - k$, which is equal to
the dimension of $\langle S \rangle \cap \langle T \rangle$,
is maybe the most influential parameter
when studying the representability
of the bi-uniform matroid over finite fields.
Observe that the value of this parameter is the same
for a bi-uniform matroid $M$ and for its dual $M^*$.
If $d = 0$, then the problem reduces to the representability of the uniform matroid.
Similarly, if $d = 1$ then, by adding to $S \cup T$ a nonzero vector
in the one-dimensional intersection of $\langle S \rangle$ and $\langle T \rangle$,
the problem again reduces to the representability of the uniform matroid.
From now on, we assume that $d = m + \ell - k \geq 2$.

\begin{proposition}
If $k \leq m+\ell-2$ and the bi-uniform matroid of rank $k$
and sub-ranks $m, \ell$ is
representable over ${\mathbb F}_q$, then $|E| \leq q+k-1$.
\end{proposition}

\begin{proof}
Take a subset $A$ of $S$ of size $k-\ell$.
Then $\langle A \rangle \cap \langle T \rangle = \{ 0 \}$
because $A \cup C$ is a basis for every
subset $C$ of $T$ of size $\ell$.
Since $k - \ell \le m - 2$, we can project
the points of $S \setminus A$ onto
$\langle S \rangle \cap \langle T \rangle$,
by defining $A'$ to be a set of $|S|-(k-\ell)$ vectors,
each a representative of a distinct $1$-dimensional subspace
$\langle x,A \rangle\cap (\langle S \rangle \cap  \langle T \rangle)$
for some $x \in (S \setminus A)$.

Let $B$ be a subset of $T$ of size $\ell-2$.
For all $x \in A'$, if $\langle B,x \rangle$ contains $\ell-1$ points of $T$
then $\langle A,B,x \rangle$ is a hyperplane of ${\mathbb F}_q^k$ containing $k$ points of $S \cup T$,
at most $m-1$ points of $S$ and $\ell-1$ points of $T$.
This cannot occur since such a set must be a basis, by hypothesis.

Thus, each of the $q+1$ hyperplanes containing $\langle B \rangle$
contains at most one vector of $A' \cup (T\setminus B)$. This gives
\(
|T|-(\ell-2)+|S|-(k-\ell) \leq q+1
\),
which gives the desired bound, since $E=S \cup T$.
\end{proof}

\begin{proposition}
If $q \leq k \leq m+\ell-2$, then the bi-uniform matroid is not representable over~${\mathbb F}_q$.
\end{proposition}

\begin{proof}
Assume that the bi-uniform matroid is representable and we have the sets of vectors $S$ and $T$ as before.
Let $e_1,\ldots,e_m$ be vectors of $S$. These vectors form a basis for $\langle S\rangle$
and we can extend them with $k - m$ vectors $e_{m+1},\ldots,e_k$ of $T$ to a basis of $\langle S,T\rangle$.
For every vector in $T$ that is not in the basis $\{e_1,\ldots, e_k \}$,
all its coordinates in this basis are non-zero.
Indeed, if there is such a vector with a zero coordinate in the $i\geq m+1$ coordinate
then the hyperplane $X_i=0$ contains $m$ vectors of $S$ and $k-m$ vectors of $T$, which does not occur.
Similarly, if the zero coordinate is in the $i\leq m$ coordinate then the hyperplane $X_i=0$
contains $m-1$ vectors of $S$ and $k-m+1$ vectors of $T$, which also does not occur.
Thus, by multiplying the vectors in the basis
by some nonzero scalars, we can assume that
$e_1 + \cdots + e_k$ is a vector of $T$
and all the coordinates of the other vectors in
$T \setminus \{e_{m+1},\ldots,e_k\}$ are non-zero.

Since $\ell \ge k - m + 2$,
there is a vector $z \in T\setminus \{e_{m+1},\ldots,e_k, e_1+ \cdots +e_k \}$.
Since $k\geq q$ there are coordinates $i$ and $j$ such that $z_i=z_j$.
If $1 \leq i \leq m$ and $1 \leq j \leq m$ then the hyperplane $X_i=X_j$
contains $m-2$ vectors of $S$ and $k-m+2 \leq \ell$ vectors of $T$, which cannot occur.
If $1 \leq i \leq m$ and $m+1 \leq j \leq k$ then the hyperplane $X_i=X_j$
contains $m-1$ vectors of $S$ and $k-m+1$ vectors of $T$, which also cannot occur.
Finally, if $m+1 \leq i \leq k$ and $m+1 \leq j \leq k$ then the hyperplane $X_i=X_j$
contains $m$ vectors of $S$ and $k-m$ vectors of $T$, which cannot occur, a contradiction.
\end{proof}

\section{Representations of the bi-uniform matroid}
\label{pt:reps}

\begin{theorem}
\label{st:rep4d2}
The bi-uniform matroid of rank $k$ and sub-ranks
$m$ and $\ell$ with $d = m + \ell - k = 2$
is representable over ${\mathbb F}_q$ if
$q$ is odd and $\max\{|E_1|,|E_2|\} \leq (q - 1)/2$.
\end{theorem}

\begin{proof}
Let $L$ denote the set of non-zero squares of ${\mathbb F}_q$ and
$(-1)^{\ell+m}\eta $ a fixed non-square of ${\mathbb F}_q$.
Consider the subsets of ${\mathbb F}_q^k$
$$
S = \{ (t,t^2,\ldots,t^{m-2},1,t^{m-1},0,\ldots,0) \ | \ t \in L \}
$$
and
$$
T = \{ (0,\ldots,0,\eta,t^{\ell-1},t^{\ell-2},\ldots,t) \ | \ t \in L \},
$$
where the coordinates are with respect to the basis $\{e_1,\ldots,e_k\}$.
We prove in the following that any
injective map which maps the elements of $E_1$
to a subset of $S$ and the elements of $E_2$ to a subset of $T$
is a representation of the bi-uniform matroid.

Since every set of $S \cup \{e_{m-1},e_{m}\}$
of size $m$ is a basis of $\langle S \rangle$,
every set formed by $m - 2$ vectors in $S$
and $\ell$ vectors in $T$ is a basis.
Symmetrically, the same holds for every
$m$ vectors in $S$ and $\ell - 2$ vectors in $T$.

The proof is concluded by showing that
there is no hyperplane $H$ of ${\mathbb F}_q^k$ containing $m - 1$ points of $S$
and $\ell - 1$ points of $T$.
Suppose that, on the contrary, such a hyperplane $H$ exists.
Since $S \cup A$ span ${\mathbb F}_q^k$ for every
$A \subseteq T$ of size $\ell - 2$,
the hyperplane $H$ intersects $\langle S \rangle$
in an $(m-1)$-dimensional subspace.
Symmetrically, $H \cap \langle T \rangle$ has dimension $\ell-1$.
Therefore, $H$ intersects
$\langle e_{m-1},e_m \rangle = \langle S \rangle \cap \langle T \rangle$
in a one-dimensional subspace.
Take elements $a_1$ and $a_2$ of ${\mathbb F}_q$,
not both zero, with $a_1 e_{m-1} + a_2 e_m \in H$.
The $m-1$ vectors of $H \cap S$ together with $a_1 e_{m-1}+a_2 e_m$ are linearly dependent.
Thus, there are $m-1$ different elements $t_1,\ldots,t_{m-1}$ of $L$ such that
$$
\det\left(\sum_{i=1}^{m-2} t_1^i e_i + e_{m-1} + t_1^{m-1} e_m, \ldots,
\sum_{i=1}^{m-2} t_{m-1}^i e_i + e_{m-1} + t_{m-1}^{m-1} e_m, a_1 e_{m-1}+a_2 e_m \right) = 0.
$$
Expanding this determinant by the last column gives
$$
a_2 (-1)^m V(t_1,\ldots,t_{m-1})=
a_1 V(t_1,\ldots, t_{m-1}) \prod_{i=1}^{m-1} t_i,
$$
where $V(t_1,\ldots,t_{m-1})$ is the determinant of the Vandermonde matrix.
Since $a_1=0$ implies $a_2=0$, we can assume that $a_1 \neq 0$ and so $a_2a_1^{-1}(-1)^{m} \in L$.
Analogously, the $\ell-1$ vectors of $H \cap T$
together with $a_1 e_{m-1} + a_2 e_m$ are linearly dependent, and hence
there are $\ell-1$ elements $u_1,\ldots,u_{\ell-1}$ of $L$ such that
$$
\det \left(\eta e_{m-1}+\sum_{i=1}^{\ell-1} u_1^i e_{k+1-i},\ldots,
\eta e_{m-1}+\sum_{i=1}^{\ell-1} u_{\ell-1}^i e_{k+1-i},a_1 e_{m-1}+a_2 e_m \right)=0.
$$
Expanding this determinant by the last column gives
$$
\eta a_2(-1)^{\ell} V(u_1,\ldots,u_{\ell-1})=
a_1 V(u_1,\ldots,u_{\ell-1}) \prod_{i=1}^{\ell-1} u_i.
$$
Since $a_1=0$ implies $a_2=0$, we can assume that $a_1 \neq 0$ and so $\eta a_2a_1^{-1}(-1)^{\ell} \in L$,
and since $a_2a_1^{-1}(-1)^{m} \in L$, this gives $\eta (-1)^{\ell+m} \in L$.
However, $\eta$ was chosen so that this is not the case.
\end{proof}

We describe in the following a family of linear evaluation codes
that will provide different representations of
the bi-uniform matroid for all possible values of the
rank $k$ and the sub-ranks $m, \ell$.
Take $\beta \in {\mathbb F}_q$ and the subspace $V$ of
${\mathbb F}_q[x] \times {\mathbb F}_q[y]$ defined by
$$
V= \{ (f(x),g(y)) \ | \ f(x)=f_1(x)+x^{m-d}g_1(\beta x),\ g(y)=g_1(y)+y^{d}g_2(y),
$$
$$
\mathrm{deg}(f_1) \leq m-d-1,\ \mathrm{deg}(g_1) \leq d-1,\  \mathrm{deg}(g_2) \leq \ell-d-1\},
$$
where $d = m + \ell - k$.
Let $F_1=\{ x_1,\ldots,x_{N_1} \}$ and
$F_2=\{ y_1,\ldots,y_{N_2} \}$ be subsets of ${\mathbb F}_q \setminus \{0\}$,
where $N_1 = |E_1|$ and $N_2 = |E_2|$.
Define $C = C(F_1,F_2,\beta)$ to be the linear evaluation code
$$
C = \{((f(x_1), \ldots, f(x_{N_1}),
g(y_1), \ldots, g(y_{N_2})) \ | \ (f,g) \in V \}.
$$
Note that $\mathrm{dim}C=\mathrm{dim}V=m-d+\ell-d+d=k$.

Every linear code determines a matroid, 
namely the one that is represented
by the columns of a generator matrix $G$,
which is the same for all generator matrices of the code.
We analyze now under which conditions
the code $C = C(F_1, F_2,\beta)$ provides a
representation over ${\mathbb F}_q$ of the bi-uniform matroid
by identifying $E_1$ and $E_2$ to $F_1$ and $F_2$, respectively
(that is, to the first $N_1$ columns and the last $N_2$ columns of $G$, respectively).

Clearly, for every $A \subseteq E$
with $|A \cap E_1| > m$ or $|A \cap E_2| > \ell$,
the corresponding columns of $G$ are linearly dependent.

Let $B$ be a basis of the bi-uniform matroid with
$|B \cap E_1| = m - t_1$ and
$|B \cap E_2| = \ell - t_2$,
where $0 \le t_i \le d$ and $t_1 + t_2 = d$.
We can assume that $B \cap E_1$ is mapped to
$\{x_1, \ldots, x_{m - t_1} \} \subseteq F_1$
and $B \cap E_2$ is mapped to
$\{y_1, \ldots, y_{\ell - t_2} \} \subseteq F_2$.
The corresponding columns of $G$ are linearly independent
if and only if $(f,g) = (0,0)$ is the only element in $V$ satisfying
\begin{equation}
\label{eq:kernel}
(f(x_1), \ldots,  f(x_{m - t_1}),
g(y_1), \ldots, g(y_{\ell - t_2})) = 0.
\end{equation}
Let
$$
r(x) = (x - x_1) \cdots (x- x_{m - t_1}) =
\sum_{i=0}^{m-t_1}r_i x^i
$$
and
$$
s(y) = (y - y_1) \cdots (y - y_{\ell - t_2}) =
 \sum_{i=0}^{\ell-t_2}s_i y^i.
$$
Then $(f,g) \in V$ satisfy~(\ref{eq:kernel}) if and only if
$f(x) = a(x)r(x)$ for some polynomial
$a(x) = \sum_{i=0}^{t_1-1} a_i x^i$ and
$g(y) = b(y)s(y)$ for some polynomial
$b(y)=\sum_{i=0}^{t_2-1} b_i y^i$.
Since $f(x) = a(x)r(x) = f_1(x)+x^{m-d}g_1(\beta x)$,
$$
g_1(\beta x) = \sum_{i=0}^{t_1-1} a_i
\left(\sum_{j= 0}^{d-t_1+i} r_{m-d+j-i}\, x^j\right),
$$
where $r_j = 0$ if $j < 0$.
On the other hand, $g(y) = b(y)s(y) = g_1(y)+y^d g_2(y)$ and so
$$
g_1(y) = \sum_{i=0}^{t_2-1} b_i \left(\sum_{j=i}^{d-1} s_{j-i}\, y^j\right),
$$
where $s_j = 0$ if $j > \ell - t_2$. Hence,
\begin{equation}
\label{eq:lincomb}
\sum_{i=0}^{t_1-1} a_i \left(\sum_{j=0}^{d-t_1+i} r_{m-d+j-i}\, x^j \right)
=\sum_{i=0}^{t_2-1} b_i\left(\sum_{j=i}^{d-1} s_{j-i} (\beta x)^j \right).
\end{equation}
If $(f,g) \neq 0$ then either $a$ or $b$ is nonzero and so there
is a linear dependence between the~$d$ polynomials in Equation~(\ref{eq:lincomb}).
Therefore, the determinant of the $d \times d$ matrix
\begin{equation}
\label{eq:bigmatrix}
\left(
\begin{array}{ccccccccc}
r_{m-d} & r_{m-d+1} & \cdots & \cdots & \cdots &r_{m-t_1} & 0 & \cdots & 0 \\
r_{m-d-1} & r_{m-d}  & \cdots & \cdots & \cdots &  \cdots & r_{m-t_1} &\cdots & 0  \\
\vdots & \vdots &  &   & & &  & \ddots  & \vdots \\
& & \cdots & \cdots & \cdots  &  \cdots & \cdots &  \cdots & r_{m-t_1} \\
\hline
s_0 & s_1 \beta & \cdots & \cdots & \cdots  & \cdots &  \cdots & \cdots &  \\
0 & s_0 \beta  & s_1 \beta^2 & \cdots  & \cdots & \cdots & \cdots & \cdots &  \\
\vdots & \ddots & \ddots & \ddots & & & &   & \vdots \\
0 & \cdots & 0 & s_0 \beta^{t_2-2} & s_1 \beta^{t_2-1}    & \cdots & \cdots  & \cdots &  s_{t_1 + 1}\beta^{d-1}\\
0 & \cdots &  \cdots & 0 & s_0 \beta^{t_2-1}  & s_1 \beta^{t_2} & \cdots  & \cdots & s_{t_1}\beta^{d-1}
\end{array}
\right)
\end{equation}
is zero.

In conclusion, the code $C(F_1,F_2,\beta)$ provides a
representation over $\mathbb{F}_q$ of the bi-uniform matroid
if and only if the determinant of the matrix~(\ref{eq:bigmatrix})
is nonzero for every choice of $m - t_1$ elements in $F_1$ and
$\ell - t_2$ elements in $F_2$ with $0 \le t_i \le d$ and $t_1 + t_2 = d$.
Clearly, this is always the case if $t_1 = 0$ or $t_2 = 0$.
Otherwise, that determinant can be expressed
as an ${\mathbb F}_{q}$-polynomial on $\beta$.
The degree of this polynomial $\varphi(\beta)$ is at most $d(d-1)/2$.
In addition, $\varphi(\beta)$ is not identically zero because
the term with the minimum power of $\beta$ is
equal to $1 \beta \cdots \beta^{t_2 - 1}  s_0^{t_2} r_{m-t_1}^{t_1}$,
and $r_{m-t_1} = 1$ and $s_0 \neq 0$.
In the next two theorems we present two different ways to select
$F_1, F_2, \beta$ with that property.

\begin{theorem}
\label{st:repextfield}
The bi-uniform matroid of rank $k$ and sub-ranks $m$ and $\ell$
with $d = m + \ell - k \ge 2$ is representable over ${\mathbb F}_q$
if $q = q_0^s$ for some $s > d(d-1)/2$ and
some prime power $q_0 > \max\{|E_1|,|E_2|\}$.
Moreover, such a representation can be obtained
in time polynomial in the size of the ground set.
\end{theorem}

\begin{proof}
Take $F_1$ and $F_2$ from
${\mathbb F}_{q_0} \setminus \{0\}$
and take $\beta \in {\mathbb F}_q$ such that its
minimal polynomial over ${\mathbb F}_{q_0}$ is of degree $s$.
The algorithm by Shoup~\cite{Sho90} finds
such a value $\beta$ in time polynomial in $q_0$ and $s$.
Then the code $C(F_1,F_2,\beta)$ gives a representation
over ${\mathbb F}_{q}$ of the bi-uniform matroid.
Indeed, all the entries in the matrix~(\ref{eq:bigmatrix}), except the
powers of $\beta$, are in ${\mathbb F}_{q_0}$.
Therefore, $\varphi(\beta)$ is a nonzero ${\mathbb F}_{q_0}$-polynomial
on $\beta$ with degree smaller than $s$.
\end{proof}

Our second construction of a code $C(F_1,F_2,\beta)$
representing the bi-uniform matroid is done over a
prime field $\mathbb{F}_p$.
We need the following well known bound
on the roots of a real polynomial.

\begin{lemma}
\label{st:bound}
The absolute value of every root
of the real polynomial $c_0 + c_1 x + \cdots + c_n x^n$
is at most
\(
1 + \max_{0 \le i \le n-1} |c_i|/ |c_n|.
\)
\end{lemma}

\begin{theorem}
\label{st:repprimefield}
Let $M$ be the bi-uniform matroid of of rank
$k$ and sub-ranks $m$ and $\ell$ with $d = m + \ell - k \ge 2$
and $m \ge \ell$.
Take $N = \max\{|E_1|, |E_2|\}$ and $K = \lceil N/2 \rceil + 1$.
Then $M$ is representable over ${\mathbb F}_p$
for every prime  $p > K^h$,
where $h =  md(1 + d(d-1)/2)$.
Moreover, such a representation can be obtained
in time polynomial in the size of the ground set.
\end{theorem}

\begin{proof}
First, we select the value $\beta$ and the sets $F_1, F_2$
among the integers in such a way that the determinant
of the real matrix~(\ref{eq:bigmatrix}) is always nonzero.
Then we find an upper bound on the absolute
value of this determinant.
The code $C(F_1,F_2,\beta)$ will represent
the bi-uniform matroid over $\mathbb{F}_p$
if $p$ is larger than that bound.

Consider two sets of nonzero integer numbers
$F_1, F_2$ with $|F_i| = |E_i|$ in the interval $[-(K-1),K-1]$.
Take $m - t_1$ values in $F_1$ and $\ell - t_2$ values in $F_2$,
where $1 \le t_i \le d-1$ and $t_1 + t_2 = d$.
Then the values $r_i$ appearing in the
matrix~(\ref{eq:bigmatrix}) satisfy
\[
|r_{m - t_1 - i}| \le {m - t_1 \choose i} (K-1)^i
\]
for every $i = 0, \ldots, m - t_1$,
and hence
\(
\sum_{i = 0}^{m-t_1} |r_i| \le K^{m-t_1}.
\)
Analogously,
\(
\sum_{i = 0}^{\ell-t_2} |s_i| \le K^{\ell-t_2}.
\)
Since $r_{m-t_1} = s_{\ell - t_2} = 1$ and $m \ge \ell$,
all values $|r_i|$, $|s_j|$ are less than or equal to $K^m - 1$.
Then $\varphi(\beta)$ is a real
polynomial on $\beta$ with degree at most $d(d-1)/2$ such that the absolute value
of every coefficient is at most $(K^m - 1)^d  < K^{md} - 1$.
Take $\beta = K^{md}$.
By Lemma~\ref{st:bound}, $\varphi(\beta) \ne 0$.
Moreover,
\[
|\varphi (\beta)| \le (K^{m} - 1)^{d} \;
\frac{\beta^{d(d-1)/2+1} - 1}{\beta -1} < K^h.
\]
Finally, consider a prime $p > K^{h}$
and reduce $\beta = K^{md}$ and the elements in $F_1$ and $F_2$ modulo~$p$.
The code $C(F_1,F_2,\beta)$ represents
the bi-uniform matroid $M$ over $\mathbb{F}_p$.
Observe that the number of bits that are needed
to represent the elements in $\mathbb{F}_p$
is polynomial in the size of the ground set.
\end{proof}



\begin{thebibliography}{99}

\bibitem{Ball2011}
{\sc S. Ball},
{\em On sets of vectors of a finite vector space in which every subset of basis size is a basis},
J. Eur. Math. Soc.,  14 (2012), pp.~733--748.

\bibitem{BTW08}
{\sc A. Beimel, T. Tassa, and E. Weinreb},
{\em Characterizing ideal weighted 
threshold secret sharing},
SIAM J. Discrete Math., 22 (2008), pp.~360--397.

\bibitem{BeWe93}
{\sc A. Beutelspacher and F. Wettl},
{\em On $2$-level secret sharing},
Des. Codes Cryptogr., 3 (1993), pp.~127--134.

\bibitem{Bri89}
{\sc E.~F. Brickell},
{\em Some ideal secret sharing schemes},
J. Combin. Math. Combin. Comput.,
9 (1989), pp.~105--113.

\bibitem{CDGJLMP08}
{\sc R. Cramer, V. Daza, I. Gracia, J. Jim\'enez Urroz,
G. Leander, J. Mart\'{\i}-Farr\'e, and C. Padr\'o},
{\em On codes, matroids and secure multi-party 
computation from linear secret sharing schemes},
IEEE Trans. Inform. Theory,
54 (2008), pp.~2644--2657.

\bibitem{FMP07}
{\sc O. Farr\`as, J. Mart\'{\i}-Farr\'e, and C. Padr\'o},
{\em Ideal multipartite secret sharing schemes},
J. Cryptology, 25 (2012), pp.~434--463.


\bibitem{FaPa11}
{\sc O. Farr\`as and C. Padr\'o},
{\em Ideal secret sharing schemes for 
useful multipartite access structures},
in International Workshop on 
Coding and Cryptology IWCC2011,
LNCS, vol. 6639,  
Y.~M Chee, Z. Guo, S. Ling,  
F. Shao, Y. Tang, H. Wang, C. Xing, eds., 
Springer, Heidelberg,  2011,
pp.~99-108. 



\bibitem{FaPa10}
{\sc O. Farr\`as and C. Padr\'o},
{\em Ideal hierarchical secret sharing schemes},
IEEE Trans. Inform. Theory,
58 (2012), pp.~3273--3286.

\bibitem{FPXY11}
{\sc O.~Farr\`as, C.~Padr\'o,
C. Xing, and A. Yang},
{\em Natural generalizations of 
threshold secret sharing},
in Advances in Cryptology, Asiacrypt 2011,
LNCS, vol. 7073,  D.~H. Lee and X. Wang, eds., 
Springer, Heidelberg, 2011, pp.~610--627.


\bibitem{GOVW98}
{\sc J. Geelen, J. Oxley, D. Vertigan, and G. Whittle},
{\em Weak maps and stabilizers of classes of matroids},
Adv. in Appl. Math.,
21 (1998), pp.~305--341.

\bibitem{HeSa06}
{\sc J. Herranz and G. S\'aez},
{\em New results on multipartite access structures},
IET Proceedings of Information Security,
153 (2006), pp.~153--162.

\bibitem{HS2001}
{\sc J.~W.~P. Hirschfeld and L. Storme},
{\em The packing problem in statistics,
coding theory and finite projective spaces: update 2001},
in Developments in Mathematics, 3, Finite Geometries,
Proceedings of the  Fourth Isle of Thorns Conference, 
Kluwer Academic Publishers,
pp.~201--246.

\bibitem{MaPa10}
{\sc J. Mart\'{\i}-Farr\'e and C. Padr\'o},
{\em On secret sharing schemes, 
matroids and polymatroids},
J. Math. Cryptol., 4 (2010), pp.~95--120.

\bibitem{Ng03}
{\sc S.~L. Ng},
{\em A representation of a family 
of secret sharing matroids},
Des. Codes Cryptogr., 30 (2003), pp.~5--19.

\bibitem{Ng06}
{\sc S.~L. Ng},
{\em Ideal secret sharing schemes with 
multipartite access structures},
IEE Proc. Commun., 153 (2006) pp.~165--168.

\bibitem{NgWa01}
{\sc S.~L. Ng and M. Walker},
{\em On the composition of matroids 
and ideal secret sharing schemes},
Des. Codes Cryptogr. 24 (2001), pp.~49--67.

\bibitem{Oxl92}
J.~G. Oxley,
{\em Matroid Theory},
Oxford Science Publications,
The Clarendon Press, 
Oxford University Press, New York, 1992.

\bibitem{PaSa00}
{\sc C. Padr\'o and G. S\'aez},
{\em Secret sharing schemes with 
bipartite access structure},
IEEE Trans. Inform. Theory,
46 (2000), pp.~2596--2604.

\bibitem{Sha79}
{\sc A. Shamir},
{\em How to share a secret},
Commun. of the ACM,
22 (1979), pp.~612--613.

\bibitem{Sho90}
{\sc V. Shoup},
{\em New algorithms for finding irreducible
polynomials over finite fields},
Math. Comp.,
54 (1990), pp.~435--447.

\bibitem{Sim88}
{\sc G.~J. Simmons},
{\em How to (really) share a secret},
in Advances in Cryptology, CRYPTO'88,
LNCS, vol. 403, S. Goldwasser, ed., 
Springer, Heidelberg, 1990, pp.~390--448.

\bibitem{Tas07}
{\sc T. Tassa},
{\em Hierarchical threshold secret sharing},
J. Cryptology, 20 (2007), pp.~237--264.

\bibitem{TaDy06}
{\sc T. Tassa and N. Dyn},
{\em Multipartite secret sharing by 
bivariate interpolation},
J. Cryptology, 22 (2009), pp.~227--258.

\end{thebibliography}
 \end{document}